\documentclass[11pt,reqno]{amsart}
\usepackage{graphicx}
\usepackage{amssymb,amsmath}
\usepackage{amsthm}
\usepackage{color,graphicx}
\usepackage{hyperref}
\usepackage{color}
\usepackage{epstopdf}
\usepackage{caption}
\usepackage{subcaption}
\usepackage[dvips]{epsfig}

\newcommand{\be}{\begin{equation}}
\newcommand{\ee}{\end{equation}}

\newcommand{\dn}{{\rm \,dn}}

\newcommand{\R}{{\mathbb R}}

\newcommand{\ve}{{\varepsilon}}

\numberwithin{equation}{section}
\numberwithin{figure}{section}

\newtheorem{theorem}{Theorem}[section]
\newtheorem{proposition}[theorem]{Proposition}

\newtheorem{lemma}[theorem]{Lemma}
\newtheorem{corollary}[theorem]{Corollary}
\newtheorem{definition}[theorem]{Definition}

\begin{document}
\vglue-1cm \hskip1cm
\title[Modified Camassa-Holm Equation]{A comment about the paper On the instability of elliptic traveling wave solutions of the modified Camassa-Holm equation}

\begin{center}



\maketitle

{\bf Renan H. Martins}

{Departamento de Matem\'atica - Universidade Estadual de Maring\'a\\
Avenida Colombo, 5790, CEP 87020-900, Maring\'a, PR, Brazil.}\\
{r3nan$\_$s@hotmail.com}

{\bf F\'abio Natali}

{Departamento de Matem\'atica - Universidade Estadual de Maring\'a\\
Avenida Colombo, 5790, CEP 87020-900, Maring\'a, PR, Brazil.}\\
{fmanatali@uem.br}

\end{center}

\begin{abstract}
We discuss the recent paper by A. Dar\'os and L.K. Arruda (On the instability of elliptic traveling wave solutions of the modified Camassa–Holm equation, J. Diff. Equat., 266 (2019), 1946-1968). Our intention is to correct some imperfections left by the authors and present the orbital stability of periodic snoidal waves in the energy space.
\end{abstract}

\section{Introduction}
The main goal of this note is to present a different result as determined in \cite{DLK} concerning the orbital stability of smooth periodic waves associated with the modified Camassa-Holm equation given by
\be\label{mCH}
u_t-u_{txx}=uu_{xxx}+2u_xu_{xx}-3u^2u_{x},
\ee
where $u:\mathbb{R}\times\R\to\R$ is a real function which is $L-$periodic at space variable.\\
\indent Formally, equation $(\ref{mCH})$ admits the conserved quantities
\begin{equation}\label{Eu}
E(u)=-\int_{0}^{L}\left[\frac{u^4}{4}+\frac{uu_x^2}{2}\right]dx,
\end{equation}

\begin{equation}\label{Fu}
F(u)=\frac{1}{2}\int_{0}^{L} u^2+u_x^2dx,
\end{equation}
and 
\begin{equation}\label{Vu}
V(u)=\int_{0}^{L} udx.
\end{equation}

A smooth periodic traveling wave solution for \eqref{mCH} is a solution of the form $u(x,t)=\phi(x-c t)$, where $c$ is a positive real constant representing the wave speed and $\phi:\mathbb{R}\to\R$ is a smooth $L-$periodic function satisfying 
$\phi^{(n)}(x+L)=\phi^{(n)}(x)$ for all $n\in\mathbb{N}$. Substituting this form into (\ref{mCH}), we obtain 
\begin{equation}\label{ode-wave}
(\phi-c)\phi''+\frac{\phi'^2}{2}-\phi^3+c\phi=A
\end{equation}
where $\phi_{c}:=\phi$ and $A_c:=A$ depend both on $c>0$. $A$ is a constant of integration.

In view of the conserved quantities $(\ref{Eu})$, $(\ref{Fu})$ and $(\ref{Vu})$, we may define the augmented Lyapunov functional,
\begin{equation}\label{lyafun}
G(u)=E(u)+cF(u)-AV(u),
\end{equation}
and the symmetric linearized operator around the wave $\phi$ expressed by
\begin{equation}\label{operator}
\mathcal{L}=G''(\phi)=(\phi-c)\partial_x^2+\phi'\partial_x+c-3\phi^2+\phi''.
\end{equation}
In addition, it is clear from $(\ref{ode-wave})$ that $G'(\phi)=0$.\\
\indent In \cite{DLK} the authors have been established the existence of periodic waves with the zero mean property associated with the equation $(\ref{ode-wave})$. They put forwarded an explicit solution given in terms of the Jacobi Elliptic Function with \textit{snoidal} type given by
\begin{equation}\label{snoidal1}
\phi(x)=\alpha+\beta{\rm sn}^2\left(\frac{2K(k)x}{L};k\right),
\end{equation}
where $\alpha$ and $\beta$ are smooth functions depending on the period $L>0$ (which needs to be large enough) and the modulus $k\in(0,1)$. Here $K:=K(k)$ represents the complete elliptic integral of first kind.\\
\indent The results contained in \cite[Theorem 2]{DLK} do not bring any mention if the periodic wave has zero mean. In addition, the authors should use the implicit function theorem taking account this property in order to prove the existence of a smooth curve of periodic waves. This fact has been determined first in \cite{ABS} where the authors  constructed snoidal periodic waves with zero mean and depending smoothly on the wave speed $c$ for the standard Korteweg-de Vries equations (in fact, they constructed cnoidal periodic waves, but to get the elliptic function depending on snoidal it makes necessary to use the basic equality $sn^2+cn^2=1$).\\
\indent  On the other hand, the standard equality concerning Jacobi Elliptic Functions given by $k^2sn^2+dn^2=1$ can be used to deduce from $(\ref{snoidal1})$, a convenient solution given in terms of the Jacobi Elliptic Function with \textit{dnoidal} type as, 
\begin{equation}\label{dnoidal1}
\phi(x)=a+b\left({\rm dn}^2\left(\frac{2K(k)x}{L};k\right)-\frac{E(k)}{K(k)}\right),
\end{equation}
where $E$ is the complete elliptic integral of second kind. The main advantage of the formula $(\ref{dnoidal1})$ is that 
$\frac{1}{L}\int_{0}^L\phi dx=a.$\\
\indent Substituting the solution $(\ref{dnoidal1})$ (or equivalently, $(\ref{snoidal1})$) in $(\ref{ode-wave})$, we obtain thanks to the terms $\phi'^2$ and $\phi^3$, a complicated equation given in short by
\begin{equation}\label{powerSN}
\sum_{i=0}^{3}f_i(k,a,b,L,c,A){\dn}^{2i}\left(\frac{2K(k)x}{L};k\right)=0,
\end{equation}
where $f_i$, $i=0,1,2,3$, are smooth functions depending on the variables $k,b,a,L,c$, and $A$. Using again the equality $k^2sn^2+dn^2=1$, we get a similar equality as in $(\ref{powerSN})$ with $sn$ instead of $dn$. Since our intention is to get a smooth curve of periodic waves depending on the modulus $k\in(0,1)$ with fixed period $L>0$, we need to consider $f_i\equiv0$, $i=0,1,2,3$, to get $a,b,c$, and $A$ in terms of $k$ and $L$. This can be done and we have
\be\label{a}
\begin{array}{llll}a&=&\displaystyle-\frac{1}{3L^2}\left[-32(2-k^2)K(k)^2+96E(k)K(k)
	+\frac{3}{2}L^2\right.\\\\
&-&	\displaystyle\left.\frac{1}{2}\sqrt{9L^4-2048K(k)^4+2048K(k)^4k^2-2048K(k)^4k^4}\right],
\end{array}
\ee

\be\label{b}
b=-\frac{32K(k)^2}{L^2},
\ee
and

\be\label{c}
c=\frac{\frac{3}{2}L^2-\frac{1}{2}\sqrt{9L^4-2048K(k)^4+2048K(k)^4k^2-2048K(k)^4k^4}}{L^2}.
\ee
The common term present in the square root appearing in equalities $(\ref{a})$ and $(\ref{c})$ gives us that the period $L$ needs to be considered large enough. In addition, the value of $c$ in $(\ref{c})$ is the same as in \cite{DLK}.\\
\indent It is clear from $(\ref{a})$ that if $\phi$ enjoys the zero mean property, one sees that $a=0$ and, in this case, $L$ can be seen as an implicit function in terms of the modulus $k\in(0,1)$. When the period is a function depending on the modulus $k$ (as a consequence, $L$ depends on the wave speed $c$), we get additional difficulties to apply classical arguments as in \cite{bona2}, \cite{bona}, \cite{DK}, and \cite{grillakis1} to conclude orbital stability/instability results. Indeed, we need to calculate the second derivative in terms of $c$ (consequently, in terms of $k$) associated with the one parameter function $d(c)=E(\phi)+cF(\phi)$ with the period $L$ depending on the modulus $k\in(0,1)$. The arguments in \cite{DLK} have showed the orbital instability of periodic waves with fixed periods only.\\
\indent As a consequence, since the periodic wave does not have necessarily zero mean, the orbital stability/instability can not be measured by analyzing only the sign of $d''(c)$, even though we have in hands, good spectral properties for the linearized operator $\mathcal{L}$ (namely, one negative eigevalue which is simple and zero being a simple eigenvalue with associated eigenfunction $\phi'$).\\
\indent Despite of the arguments established in \cite{DLK}, we prove that the periodic wave in $(\ref{dnoidal1})$ is orbitally stable in the energy space $H_{per}^1([0,L])$ without assuming that $\phi$ has zero mean. To do so, we employ the recent development in \cite{ANP} which gives a wide approach to deduce orbital stability results regarding periodic waves $\phi$ which are, at the same time, critical points and zero solutions of the modified Lyapunov function 
\be
\mathcal{B}(u)=G(u)-G(\phi)+N(Q(u)-Q(\phi))^2,
\ee
where $N$ is a convenient positive constant and $Q$ is a convenient sum of the quantities $F(\phi)$ and $V(\phi)$. In some sense, our result recovers the orbital stability arguments as in \cite{hakka}.\\
\indent Our paper is organized as follows: next section is used to present the orbital stability of periodic waves associated with the model $(\ref{mCH})$, using the arguments in \cite{ANP}. Some important remarks concerning the orbital instability of periodic waves with zero mean are given in Section 3.

\section{Orbital Stability of Traveling Waves}\label{OSPW1}

In this section, we present our stability results by using an simplification of the arguments in \cite{ANP}. First of all, in order to simplify the notation, the norm and inner product in $L_{per}^2([0,L])$ will be denoted by  $||\cdot||$ and $\langle\cdot,\cdot\rangle$, respectively.

Before stating our main theorem, we need some preliminary results. Let $\rho$ be the semi-distance defined on the energy space $X=H_{per}^1([0,L])$.
\be\label{rho}
\rho(u,\phi)=\inf_{y\in\mathbb{R}}||u-\phi(\cdot+y)||_{X}.
\ee

\begin{definition}\label{defstab}
	We say that a solitary wave solution $\phi$ is orbitally stable in $X$, by the flow of \eqref{mCH},  if for any $\ve>0$ there exists $\delta>0$ such that for any $u_0\in X$ satisfying $\|u_0-\phi\|_X<\delta$, the solution $u(t)$ of \eqref{mCH} with initial data $u_0$ exists globally and satisfies
	$
	\rho(u(t),\phi)<\ve,
	$
	for all $t\geq0$.
\end{definition}
According with the arguments in \cite{hakka}, the Cauchy problem associated with the equation $(\ref{mCH})$ is locally well-posed in $H_{per}^s([0,L])$, for $s>\frac{3}{2}$. On the other hand, it is well known that some blow up results in finite time are expected for the same model (see again \cite{hakka} and references therein). Thus, since $F$ in $(\ref{Fu})$ is a conserved quantity, we can combine the local solution obtained in \cite{hakka} with the standard a priori estimate $F(u(t))=F(u_0)$ for all $t\geq0$ to get a \textit{conditional orbital stability result}.

For a given $\varepsilon>0$, we define the $\varepsilon$-neighborhood of the orbit $O_\phi=\{\phi(\cdot+y), y\in\R\}$ as
\be \label{tube}U_{\varepsilon} := \{u\in X;\ \rho(u,\phi) < \varepsilon\}.\ee

To start our analysis, we first assume the existence of a smooth functional $Q:X\rightarrow \mathbb{R}$ which is conserved quantity in time, invariant by translations in the sense that $Q(u(\cdot+r))=Q(u)$, for all $r\in\mathbb{R}$, and satisfying $\langle Q'(\phi),\phi'\rangle=0$. Functional $Q$ plays an important role in our analysis since it inspires us the definition of the tangent space $\Upsilon_0=\{u\in X;\ \langle Q'(\phi),u\rangle=0\}$. In addition, $Q$ will be considered with a convenient form later on.

Before starting with the stability results, we need to prove some auxiliary results which are useful in our stability analysis. The first one concerns the existence of periodic waves for large periods $L>0$. In some sense, this fact has already presented in the introduction and for the sake of completeness we enunciate a full result of existence of periodic solutions.\\
\begin{lemma}\label{lema12}
	For $L>0$ sufficiently large, there exists $k_{1}\in(0,1)$ such that for all $k\in(0,k_1)$, the periodic traveling wave solution $\phi$ in $\ref{dnoidal1}$ is a solution of $(\ref{ode-wave})$ and depending smoothly on $k\in(0,k_1)$. Parameters $a$, $b$ and $c$ are given respectively by $(\ref{a})$, $(\ref{b})$ and $(\ref{c})$. The value of $A$ in $(\ref{ode-wave})$ can be expressed in terms of $k$ and $L$ by
	\be
	\begin{array}{lllll}A&=&\displaystyle\frac{1}{27L^6}\left[(1280(-1+k^2-k^4)K(k)^4\right.\\\\
	&+&	\displaystyle\left.9L^4)
	\sqrt{2048(-1+k^2-k^4)K(k)^4+9L^4}\right.\\\\
	&+&\displaystyle\left.(-16384-16384k^6+24576k^2+24576k^4)K(k)^6\right.\\\\
	&+&\displaystyle\left. 6912L^2(1-k^2+k^4)K(k)^4-27L^6\right].\end{array}
	\ee
	Moreover, for $k\in(0,k_1)$ one has:\\
	i) the strict inequality $c^2-3c+\frac{32\pi^4}{L^4}<0$ is always satisfied.\\
	ii) $\phi-c<0$ in $[0,L]$.
\end{lemma}

\begin{flushright}
	$\square$
\end{flushright}

The next result allows us to decide about the quantity and multiplicity of the first two negative eigenvalues of $\mathcal{L}$ in $(\ref{operator})$.\\

\begin{lemma}\label{lema123}
Assume that conditions in Lemma $\ref{lema12}$ are satisfied. The linearized operator $\mathcal{L}$ defined in $(\ref{operator})$ has only one negative eigenvalue which is simple. Zero is a simple eigenvalue whose associated eigenfunction is $\phi'$.
\end{lemma}
\begin{proof}
	See Propositions $2$ and $3$ in \cite{DLK}.
	
\end{proof}

We are in position to establish the following proposition which gives a sufficient condition for the positiveness of the quadratic form associated with $\mathcal{L}$.

\begin{proposition}\label{prop2}
Suppose that conditions in Lemma $\ref{lema12}$ are satisfied. Assume the existence of $\Phi\in X$ such that $\langle\mathcal{L}\Phi,\varphi\rangle=0$, for all $\varphi\in \Upsilon_0$, and
$$
\mathcal{I}:=\langle\mathcal{L}\Phi,\Phi\rangle<0\ \ \ \ \ \ \mbox{(Vakhitov-Kolokolov condition)}.
$$
Then, there exists a constant $c>0$ such that
$\langle\mathcal{L}v,v\rangle\geq c||v||_{X}^2,$
for all $v\in \Upsilon_0\cap [\phi']^\perp$.
\end{proposition}
\begin{proof}
See \cite{ANP}.
\end{proof}

In our context, parameters $c$ and $A$ given by Lemma $\ref{lema12}$ depend on a third parameter $k\in(0,k_1)$. Next result gives us a sufficient condition to obtain a convenient formula for $\mathcal{I}$ in Lemma $\ref{prop2}$ in terms of $k$.

\begin{corollary}\label{coro2}
 Suppose that the assumptions in Lemma $\ref{lema12}$ are satisfied. The quantity $\mathcal{I}$ in Proposition $\ref{prop2}$ is given by:
\begin{equation}\label{solcriterio}
\mathcal{I}=\frac{d A}{d k} \frac{d}{dk}V(\phi)-\frac{d c}{d k}\frac{d}{dk}F(\phi).
\end{equation}
Moreover, $\mathcal{I}<0$.
\end{corollary}
\begin{proof}
In Proposition $\ref{prop2}$ we define $\Phi=\frac{d}{dk}\phi$ and $Q(u)=\frac{d A}{d k}M(u)-\frac{dc}{dk}F(u)$ to get $(\ref{solcriterio})$. Now, we to need check that $\mathcal{I}<0$. Indeed, since $V(\phi)=aL$, we obtain from $(\ref{Fu})$ and $(\ref{solcriterio})$ that
\begin{equation}\label{solcrit1}
\mathcal{I}=\frac{dA}{dk}\frac{da}{dk}L-\frac{1}{2}\frac{dc}{dk}\frac{d}{dk}\left(\int_0^{L}\phi'^2+\phi^2dx\right)
\end{equation}
The right-hand side of $(\ref{solcrit1})$ is a complicated function depending on $k\in(0,k_1)$ and $L>0$ large enough. We can show that $\mathcal{I}<0$ by plotting some pictures.

\begin{figure}[h!]
	\centering
	\begin{subfigure}[t]{0.5\textwidth}
		\centering
		\includegraphics[width=1.0\linewidth]{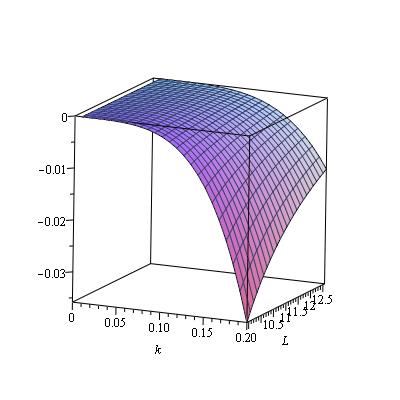}
	\end{subfigure}%
	\begin{subfigure}[t]{0.5\textwidth}
		\centering
		\includegraphics[width=1.0\linewidth]{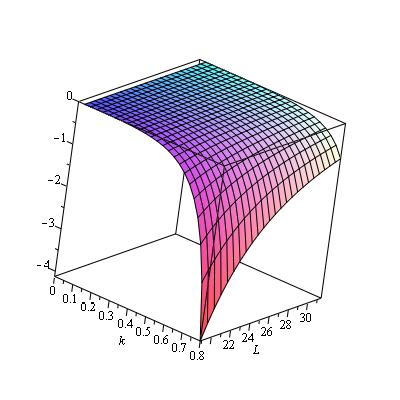}
	\end{subfigure}%
	\caption{Graphics of $\mathcal{I}$ for $(k,L)\in(0,0.2]\times[3\pi,6\pi]$ (left) and for $(k,L)\in(0,0.8]\times[6\pi,10\pi]$ (right).}
\end{figure}

\end{proof}

 Proposition \ref{prop2} and Corollary $\ref{coro2}$ are useful to establish  the following result.

\begin{lemma}\label{lemma1}
 Under assumptions of Proposition $\ref{prop2}$, there exist $N>0$ and $\tau>0$ such that
$$\langle\mathcal{L}v,v\rangle +2N\langle Q'(\phi),v\rangle^{2}\geq \tau||v||_{X}^2,$$
for all $v\in [\phi']^\perp$.
\end{lemma}
\begin{proof}
First, since $\langle Q'(\phi),\phi'\rangle=0$, we can write $v\in [\phi']^{\perp}$ as
$$v=\zeta w+z,$$
where $w=\frac{Q'(\phi)}{||Q'(\phi)||}$, $\zeta=\langle v,w\rangle$ and $z\in \Upsilon_0$. Because $z\in\Upsilon_0\cap [\phi']^{\perp}$, Proposition \ref{prop2} implies
\begin{equation}\label{eq01}
\langle\mathcal{L}v,v\rangle \geq \zeta^2\langle\mathcal{L}w,w\rangle + 2\zeta\langle\mathcal{L}w,z\rangle +C||z||_X^2.
\end{equation}
Using Cauchy-Schwartz and Young's inequalities, we have
\begin{equation}\label{eq02}
2\zeta\langle\mathcal{L}w,z\rangle \leq \frac{C}{2}||z||_X^2 + \frac{2\zeta^2}{C}||\mathcal{L}w||_X^2.
\end{equation}
Furthermore, we may choose $N>0$ such that
\begin{equation}\label{eq03}
\langle\mathcal{L}w,w\rangle -\frac{2}{c}||\mathcal{L}w||_X^2 +2N||Q'(\phi)||_X^2\geq \frac{C}{2}.
\end{equation}
We point out that $N$ depends only on $\phi$.
Therefore, using (\ref{eq01}), (\ref{eq02}) and (\ref{eq03}), we conclude
\begin{eqnarray}
\langle\mathcal{L}v,v\rangle +2N\langle Q'(\phi),v\rangle^{2}&=&\langle\mathcal{L}v,v\rangle + 2N\zeta^2||Q'(\phi)||_X^2 \nonumber \\
&\geq&\frac{C}{2}(\zeta^2+||z||_X^2) \nonumber \\
&=& \frac{C}{2}||v||_{X}^2. \nonumber
\end{eqnarray}
The proof is thus completed.
\end{proof}

Let $N>0$ be the constant obtained in the previous lemma. We define the functional $\mathcal{B}:X\rightarrow\R$ as
\begin{equation}\label{functionalV1}
\mathcal{B}(u)=G(u)-G(\phi)+N(Q(u)-Q(\phi))^2,
\end{equation}
where $G$ is the augmented functional defined in (\ref{lyafun}) and $Q$ is the functional defined in Corollary $\ref{coro2}$. It is easy to see from $(\ref{functionalV1})$ and $(\ref{ode-wave})$ that $\mathcal{B}(\phi)=0$ and $\mathcal{B}'(\phi)=0$. In addition, since $G$ is a conserved quantity, $Q$ is also a conserved quantity and the Cauchy problem related to the equation $(\ref{mCH})$ is conditionally global well-posed in the energy space $X$, one has
\begin{equation}\label{boundV}
 \mathcal{B}(u(t))=\mathcal{B}(u_0),\ \ \ \ \ \ \mbox{for all}\ t\geq0.
\end{equation}
Thus, $\mathcal{B}(u(t))$ is finite for large values of $t$.\\

\begin{lemma}\label{lemma2}
There exist $\alpha>0$ and $D>0$ such that 
\begin{equation}\label{eq003}
\mathcal{B}(u)\geq D\rho(u,\phi)^2\end{equation}
for all $u\in U_{\alpha}$
\end{lemma}
\begin{proof}
First, note that from the definition of $\mathcal{B}$ it follows that
$$\langle \mathcal{B}''(u)v,v\rangle=\langle G''(u)v,v\rangle+2N(Q(u)-Q(\phi)) \langle Q''(u)v,v\rangle+2N\langle Q'(u),v\rangle^2,$$
for all $u,v\in X$. In particular, 
$$\langle \mathcal{B}''(\phi)v,v\rangle=\langle \mathcal{L}v,v\rangle+2N\langle Q'(\phi),v\rangle^2.$$
Consequently, from Lemma \ref{lemma1} we get
\begin{equation} \label{eq001}
\langle \mathcal{B}''(\phi)v,v\rangle\geq\tau||v||_X^2,
\end{equation}
for all $v\in (\ker(\mathcal{L}))^\perp$.

On the other hand, a Taylor expansion of $\mathcal{B}$ around $\phi$  reveals that
\begin{equation}\label{eq002}
\mathcal{B}(u)=\mathcal{B}(\phi)+ \langle \mathcal{B}'(\phi),u-\phi\rangle+\frac{1}{2} \langle \mathcal{B}''(\phi)(u-\phi),u-\phi\rangle +h(u),
\end{equation}
where $\lim\limits_{u\to\phi}\frac{|h(u)|}{||u-\phi||_X^2}=0$.
Thus, we can choose $\alpha_1>0$ such that
\begin{equation}\label{limit1}|h(u)|\leq\frac{\tau}{4}||u-\phi||_X^2, \qquad \mbox{for all}  \ u\in B_{\alpha_1}(\phi),\end{equation}
 where $B_{\alpha_1}(\phi)=\left\{u\in X; ||u-\phi||_X <\alpha_1 \right\}$.
 
 Since $\mathcal{B}(\phi)=0$ and $\mathcal{B}'(\phi)=0$, we have from (\ref{eq001}), (\ref{eq002}) and $(\ref{limit1})$ that
 \begin{equation}\label{eq0031}
 \mathcal{B}(u)\geq \frac{\tau}{4}\rho(u,\phi)^2,
 \end{equation}
 for all $u\in B_{\alpha_1}(\phi)$ such that $(u-\phi)\in [\phi']^{\perp}.$\\
 \indent There exists a continuously differentiable map $r:U_{\alpha}\rightarrow\mathbb{R}$, such that  $\langle u(\cdot -r(u)),\phi'\rangle=0$ and $r(\phi)=0$, for all $u\in B_{\alpha}(\phi)$. In fact, let us define the smooth map $S:X\times \mathbb{R}\rightarrow\mathbb{R}$ given by $S(u,r)=\langle u(\cdot-r),\phi'\rangle$. Since $S(\phi,0)=\langle\phi,\phi'\rangle=0$ and 
 $S_r(\phi,0)=-\langle\phi',\phi'\rangle\neq0$, we guarantee, from the implicit function theorem, the existence of $\alpha_2>0$, an $\delta_0>0$ and a unique $C^1-$map $r:B_{\alpha_2}(\phi)\rightarrow(-\delta_0,\delta_0)$ such that $r(\phi)=0$ and $G(u,r(u))=\langle u(\cdot-r(u)),\phi'\rangle=0$, for all $u\in B_{\alpha_2}(\phi)$. From continuity arguments and since $r(\phi)=0$, for a given $0<\varepsilon\leq\min\{\alpha_1,\alpha_2\}$, there exists  $\alpha>0$ small enough (for instance, $0<\alpha\leq \varepsilon$) such that $||u(\cdot-r(u))-\phi||_X<\varepsilon$ with $\langle u(\cdot-r(u)),\phi'\rangle=0$, for all $u\in B_{\alpha}(\phi)$.\\
 \indent From $(\ref{eq0031})$ and the arguments in the last paragraph, we obtain the existence of $D>0$ such that $\mathcal{B}(u)\geq D\rho(u,\phi)^2$, for all $u\in B_{\alpha}(\phi)$. The remainder of the proof follows from similar arguments as in \cite[Lemma 4.1]{bona2}. 
\end{proof}

The above lemma is the key point to prove our main result. Roughly speaking, it says that $\mathcal{B}$ is a suitable Lyapunov function to handle with a our problem. Finally, we present our stability result.

\begin{theorem}\label{teoest11} Suppose that assumptions in Lemma $\ref{lema12}$ hold. Then $\phi$ is orbitally stable in the sense of Definition $\ref{defstab}$.
\end{theorem}
\begin{proof}
	Let $\alpha>0$ be the constant such that Lemma \ref{lemma2} holds. Since $\mathcal{B}$ is continuous at $\phi$, for a given $\varepsilon>0$, there exists $\delta\in (0,\alpha)$ such that if $||u_0-\phi||_X<\delta$  one has
	\be\label{estepsilon2}
	\mathcal{B}(u_0)\leq |\mathcal{B}(u_0)|=|\mathcal{B}(u_0)-\mathcal{B}(\phi)|<D\varepsilon^2,
	\ee
	where $D>0$ is the constant in Lemma \ref{lemma2} and $C>0$ is a constant to be presented later.\\
	\indent The continuity in time of the function $\rho(u(t),\phi)$ allows to choose $T_1>0$ such that \be\label{subalpha1}\rho(u(t),\phi)<\alpha,\ \ \  \mbox{for all}\ t\in [0,T_1).\ee
	Thus, one obtains $u(t)\in U_{\alpha}$, for all $t\in[0,T_1)$. From Lemma \ref{lemma2}, we have
	\be\label{estepsilon12}
	D\rho(u(t),\phi)^2\leq \mathcal{B}(u(t))=\mathcal{B}(u_0),\ \ \ \ \ \mbox{for all}\ t\in[0,T_1).
	\ee
	\indent Next, we finally prove that $\rho(u(t),\phi)<\alpha$, for all $t\in [0,+\infty)$, from which one concludes the orbital stability. Indeed, let  $T_2>0$ be the supremum of the values of $T_1>0$ for which $(\ref{subalpha1})$ holds. To obtain a contradiction, suppose that $T_2<+\infty$.  By choosing $\varepsilon<\frac{\alpha}{2}$ we obtain, from  $(\ref{estepsilon2})$ and $(\ref{estepsilon12})$ that
	$$
	\rho(u(t),\phi)<\frac{\alpha}{2}, \ \ \ \ \ \mbox{for all}\ t\in[0,T_2).
	$$
	Since $t\in(0,+\infty)\mapsto\rho(u(t),\phi)$ is continuous, there is $T_3>0$ such that
	$\rho(u(t),\phi)<\frac{3}{4}\alpha<\alpha$, for $t\in [0,T_2+T_3)$, contradicting the maximality of $T_2$. Therefore, $T_2=+\infty$ and the theorem  is established.
\end{proof}

\section{Remarks on the Orbital Instability of Periodic Waves with Zero Mean}

\indent In this section, we present some remarks concerning the orbital instability associated with a general Hamiltonian equation given by

\be\label{hamilt}
u_t=JE'(u),
\ee  
where $J$ is a skew-symmetric linear operator, $E$ is the energy functional related to the model and $E'$ represents the Fr\'echet derivative of $E$. We restrict ourselves to the case of the general equation 
\be\label{gCH}
u_t+(p(u))_x-u_{xxt}=\left(q'(u)\frac{u_x^2}{2}+q(u)u_{xx}\right)_x,
\ee
but our arguments can be extended for other equations. Here, $p$ and $q$ are smooth real functions with $p(0)=0$. When $E$ indicates the energy functional associated with the equation $(\ref{gCH})$ and $J=\partial_x(1-\partial_x^2)^{-1}$, particular cases of equation $(\ref{gCH})$ can be expressed as $(\ref{hamilt})$. In particlar, when $p(u)=u^3$ and $q(u)=u$, the general equation $(\ref{gCH})$ reduces to $(\ref{mCH})$ and it is possible to recover $(\ref{hamilt})$ in this specific case.\\
\indent It is well known that the classical instability theory as the one in \cite{grillakis1} can not be applied when $J$ is not one-to-one even though the eventual periodic wave of the form $u(x,t)=\phi(x-ct)$ enjoys the zero mean property (it is clear that $J=\partial_x(1-\partial_x^2)^{-1}$ is not one-to-one since $\ker(J)=[1]$). We believe that the instability analysis in \cite{grillakis1} over periodic Sobolev spaces can be done by restricting \textit{all the analysis} in the closed subspace  
\begin{equation}\label{zero}
Y_0=\Big\{f\in L^2_{\rm per}([0,L]);\ \int_{0}^{L} f(x) dx = 0 \Big\}.
\end{equation}

\indent To get a precise answer for this question, we could suggest the readers a study of a spectral instability result combined with a method where \textit{spectral instability} implies \textit{orbital instability}. As example: for the case of the generalized Korteweg-de Vries and Benjamin-Bona-Mahony equations, we  could assume that the data-solution map $u_0\mapsto u(t)$ is smooth (see \cite{AN1}). We believe that such approach can be done for the general equation $(\ref{gCH})$ without further problems.\\
\indent A convenient spectral stability criterium can be determined. In fact, we first use the perturbation $u(x,t) = \phi(x-ct) + v(x-ct,t)$ in the equation $(\ref{gCH})$ and substituting the associated equation as in (\ref{ode-wave}) for the case of equation $(\ref{gCH})$ (when possible). If everything works fine, we obtain after some calculations the standard spectral problem
\begin{equation}\label{vlinear}
v_t = \partial_x \mathcal{L} v,
\end{equation}
where $\mathcal{L}$ is the (self-adjoint) linearized operator associated with the equation $(\ref{gCH})$ around the wave $\phi$. Since $\phi$ depends only on $x$, a separation of variables in the form $v(x,t) = e^{\lambda t} \eta(x)$ with some $\lambda \in \mathbb{C}$
and $\eta : [0,L] \to \mathbb{C}$ reduces the linear equation (\ref{vlinear})
to the spectral stability problem
\begin{equation}
\label{spectral-stab}
\partial_x \mathcal{L} \eta=\lambda \eta.
\end{equation}
The spectral stability of the periodic wave $\phi$ is defined as follows.

\begin{definition}
	\label{defspe} The smooth periodic wave $\phi$ is said to be spectrally stable
	with respect to perturbations of the same period if
	$\sigma(\partial_x \mathcal{L}) \subset i\mathbb R$ in $L^2_{\rm per}([0,L])$.
	Otherwise, that is, if $\sigma(\partial_x \mathcal{L})$ in $L^2_{\rm per}([0,L])$
	contains a point $\lambda$ with $\mbox{\rm  Re}(\lambda)>0$, the periodic wave $\phi$ is said to be spectrally unstable.
\end{definition}

As we have already mentioned, we know that $\partial_x$ is not a one-to-one operator in periodic context. In order to overcome this difficult, a constrained spectral problem can be considered as
\begin{equation}\label{modspecp1}
\partial_x \mathcal{L}\big|_{Y_0}\eta=\lambda \eta,
\end{equation}
where $\mathcal{L}\big|_{Y_0}$ is
a restriction of $\mathcal{L}$ on the closed subspace $Y_0$ defined in $(\ref{zero})$. In addition, over $Y_0$, the linear operator $\partial_x$ has bounded inverse and this crucial fact could enable us to apply the arguments in \cite{grillakis1} to get orbital instability results (but it is necessary to perform suitable modifications in the mentioned theory). However, we need to analyze the quantity and multiplicity of the restricted linearized operator $\mathcal{L}\big|_{Y_0}$ instead of $\mathcal{L}$.\\
\indent To handle with $\mathcal{L}\big|_{Y_0}$, we need to count the quantity (and multiplicity) of non-positive eigenvalues associated with this restriction operator. If we assume that the kernel of $\mathcal{L}$ is simple and generated by $\phi'$, we can  use the Morse Index Formula (see \cite{Pel-book}) as
\begin{equation}\label{identnegLL}
\left\{ \begin{array}{l}
n(\mathcal{L} \big|_{Y_0}) = n(\mathcal{L}) - n(\langle \mathcal{L}^{-1}1,1\rangle) - z(\langle \mathcal{L}^{-1}1,1\rangle), \\
z(\mathcal{L} \big|_{Y_0}) = z(\mathcal{L}) + z(\langle \mathcal{L}^{-1}1,1\rangle),
\end{array} \right.
\end{equation}
 where $n(\mathcal{A})$ and $z(\mathcal{A})$ indicates, respectively, the number of negative eigenvalues (counting multiplicities) and the dimension of the kernel of a general linear operator $\mathcal{A}$. Since it has been assumed that $z(\mathcal{L})=1$, we have from $(\ref{identnegLL})$ that $z(\mathcal{L} \big|_{Y_0})=1+ z(\langle \mathcal{L}^{-1}1,1\rangle)$. Moreover, if $\langle \mathcal{L}^{-1}1,1\rangle\neq0$, one sees that $z(\mathcal{L} \big|_{Y_0})=1$ and $n(\mathcal{L} \big|_{Y_0}) = n(\mathcal{L}) - n(\langle \mathcal{L}^{-1}1,1\rangle)$.\\
 \indent Let us assume that $F$ in $(\ref{Fu})$ is a conserved quantity associated to $(\ref{gCH})$. In this setting, the main result in \cite{DK} establishes a criterium for the (spectral) orbital stability of periodic waves by using the convenient formula for the Hamltonian Krein Index as
 \be\label{krein}\mathcal{K}_{Ham}=n(\mathcal{L} \big|_{Y_0})-n(D)=n(\mathcal{L}) - n(\langle \mathcal{L}^{-1}1,1\rangle)-n(D).\ee
 The periodic wave is orbitally (spectrally) unstable if $\mathcal{K}_{Ham}=1$ and orbitally (spectrally) stable if $\mathcal{K}_{Ham}=0$. Here, $D$ is the hessian determinant associated with $F(\phi)$ and $V(\phi)$ and it needs to be non-zero. If $\phi$ has fixed period, zero mean and depends smoothly on the wave speed $c$, one has $D=-\frac{1}{2}\frac{d}{dc}\int_0^L\left(\phi'^2+\phi^2\right)dx=-d''(c)$, provided that $\langle \mathcal{L}^{-1}1,1\rangle\neq0$. Therefore, if $n(\mathcal{L})=1$ and 
 $\langle \mathcal{L}^{-1}1,1\rangle>0$, the periodic wave is (spectrally) stable if $d''(c)>0$ and (spectrally) unstable if $d''(c)<0$ by a direct application of $(\ref{krein})$. The last condition is exactly the same as requested in \cite{DLK} to conclude the orbital instability but the periodic wave $\phi$ determined by the authors only has zero mean whether $L$ depends on the modulus $k$. However, we could have $d''(c)<0$ with $\langle \mathcal{L}^{-1}1,1\rangle<0$ and using $(\ref{krein})$, we still have the (spectral) stability. In some particular cases, it is well known that if the Cauchy problem associated with the equation $(\ref{gCH})$ enjoys of a convenient global well-posedness result, the spectral stability implies the orbital stability. Therefore, \textit{in the case of a smooth curve of periodic waves $c\mapsto\phi$ with fixed period and zero mean}, we can not conclude a precise result of orbital instability only with $n(\mathcal{L})=1$, $\ker(\mathcal{L})=[\phi']$ and $d''(c)<0$ (using a combination of \cite{DK} and \cite{grillakis1}) as determined in \cite{DLK}.

\section*{Acknowledgements}

R.H.M. is supported by CAPES/Brazil [regular doctorate fellowship]. F.N. is partially supported by Funda\c{c}\~ao Arauc\'aria [grant number 002/2017] and CNPq/Brazil [grant number 304240/2018-4].

\end{document}